\documentclass{article}
\usepackage[cp1251]{inputenc}
\usepackage[russian]{babel}
\usepackage{amssymb,amsfonts,amsmath,amsthm,amscd}
\usepackage{graphicx}
\usepackage{enumerate}
\usepackage{soul}
\usepackage[matrix,arrow,curve]{xy}
\usepackage{longtable}
\usepackage{array}
\usepackage{bm}
\RequirePackage{russlh}
\RequirePackage{mathlh}

\newdimen\symskip
\newdimen\defskip
\newdimen\parind
\parind=\parindent
\newdimen\leftmarge
\newdimen\theoremshape
\topsep\defskip
\makeatletter
\newcommand*{\clei}{\nobreak\hskip\z@skip}

\renewcommand{\"}{''}

\renewcommand{\:}{\textup{:}}
\renewcommand{\~}{\textup{;}}

\DeclareRobustCommand*{\д}{\clei\hbox{-}\clei}
\newcommand{\no}{}
\renewcommand{\@listI}{\settowidth\labelwidth{\labheadi{\no}}\listipar{\parind}{\labelwidth}}
\newcommand{\listivpar}{\topsep\defskip\partopsep0pt\parsep-\parskip\itemsep0.5\topsep}
\newcommand{\listipar}[2]{\rightmargin0pt\leftmargin#1\labelsep#1\advance\labelsep-#2\itemindent0pt\listivpar}
\renewcommand{\@listii}{\settowidth\labelwidth{\labheadii{\@roman{\no}}}\listiipar{\parind}{\labelwidth}}
\newcommand{\listiivpar}{\topsep0.5\defskip\partopsep0pt\parsep-\parskip\itemsep0.5\topsep}
\newcommand{\listiipar}[2]{\rightmargin0pt\leftmargin#1\labelsep#1\advance\labelsep-#2\itemindent0pt\listiivpar}
\def\thempfn{\ifcase\value{footnote}1\or *\or **\or ***\else\@ctrerr\fi}
\renewcommand\footnoterule{%
  \kern-3\p@
  \hrule\@width1in
  \kern2.6\p@}
\makeatother

\makeatletter

\renewcommand{\@biblabel}[1]{[#1]}
\renewenvironment{thebibliography}[1]
     {\renewcommand{\refname}{References}%
      \renewcommand{\No}{}%
      \section*{\refname}%
      \@mkboth{\MakeUppercase\refname}{\MakeUppercase\refname}%
      \list{\@biblabel{\@arabic\c@enumiv}}%
           {\itemsep\baselineskip
            \leftmargin\parind
            \settowidth\labelwidth{\@biblabel{#1}}%
            \labelsep\parind\advance\labelsep-\labelwidth
            \@openbib@code
            \usecounter{enumiv}%
            \let\p@enumiv\@empty
            \renewcommand\theenumiv{\@arabic\c@enumiv}}%
      \sloppy
      \clubpenalty4000
      \@clubpenalty\clubpenalty
      \widowpenalty4000%
      \sfcode`\.\@m}
     {\def\@noitemerr
       {\@latex@warning{Empty `thebibliography' environment}}%
      \endlist}

\def\@maketitle{%
  \newpage
  \vskip0.5em%
  UDK \udk%
  \vskip0.5em%
  MSC \msc%
  \vskip1em%
  \begin{center}\bf%
  \let\footnote\thanks%
   {\Large\@author\par}%
   \vskip1.5em%
   {\LARGE\@title\par}%
   \vskip1em%
   {\large\@date}%
  \end{center}%
  \par
  \vskip1.5em}

\def\@title{\@latex@warning@no@line{No \noexpand\title given}}

\makeatother

\makeatletter

\renewcommand\sectionmark[1]{%
 \markright{%
  \ifnum \c@secnumdepth >\z@
   \thesection. \ %
  \fi
 #1}}%

\renewcommand{\section}{\@startsection{section}{1}{0pt}%
{5.5ex plus .5ex minus .2ex}{1.5ex plus .3ex}%
{\center\normalfont\Large\bfseries\sffamily\bom}}
\renewcommand{\subsection}{\@startsection{subsection}{2}{0pt}%
{4.5ex plus .4ex minus .2ex}{0.75ex plus .2ex}%
{\center\normalfont\large\bfseries\sffamily\bom}}
\renewcommand{\subsubsection}{\@startsection{subsubsection}{3}{0pt}%
{2.5ex plus .5ex minus .2ex}{1ex plus .2ex}%
{\center\normalfont\bfseries\sffamily\bom}}

\def\@postskip@{\hskip.5em\relax}

\def\postsection{.\@postskip@}

\def\postsubsection{.\@postskip@}

\def\postsubsubsection{.\@postskip@}

\def\postparagraph{.\@postskip@}

\def\postsubparagraph{.\@postskip@}

\def\@seccntformat#1{\csname pre#1\endcsname\csname the#1\endcsname\csname post#1\endcsname}

\makeatother

\renewcommand{\thesection}{\textup{\arabic{section}}}

\newcommand{\parr}{\par\addvspace{\defskip}}
\newcommand{\theo}[2]{\newtheorem{#1}{#2}[section]}
\newcommand{\deff}[2]{\newenvironment{#1}{\parr\textbf{#2.}}{\parr}}
\theo{cas}{Case}
\theo{problem}{Problem}
\theo{theorem}{Theorem}
\theo{lemma}{Lemma}
\theo{prop}{Proposition}
\theo{stm}{Statement}
\theo{fact}{Fact}
\theo{imp}{Corollary}
\theo{ex}{Example}
\deff{df}{Definition}
\deff{note}{Note}
\deff{denote}{Notation}
\deff{denotes}{Notations}
\deff{hint}{Hint}
\deff{answer}{Answer\:}

\makeatletter
\def\@begintheorem#1#2[#3]{%
  \deferred@thm@head{\the\thm@headfont \thm@indent
    \@ifempty{#1}{\let\thmname\@gobble}{\let\thmname\@iden}%
    \@ifempty{#2}{\let\thmnumber\@gobble}{\let\thmnumber\@iden}%
    \@ifempty{#3}{\let\thmnote\@gobble}{\let\thmnote\@iden}%
    \thm@notefont{\bfseries\upshape}%
    \indent%
    \thm@swap\swappedhead\thmhead{#1}{#2}{#3}%
    \the\thm@headpunct
    \thmheadnl 
    \hskip\thm@headsep
  }%
  \ignorespaces}
\renewenvironment{proof}{\setcounter{cas}{0}\parr\pushQED{\qed}\normalfont$\square\quad$}{\setcounter{cas}{0}\popQED\@endpefalse\parr}

\makeatother

\newcommand{\labheadi}[1]{\textup{#1)}}
\newcommand{\labheadii}[1]{\textup{(#1)}}

\newenvironment{nums}[1]{\renewcommand{\no}{#1}\begin{enumerate}}{\end{enumerate}}
\newcommand{\eqn}[1]{\begin{equation}#1\end{equation}}
\newcommand{\equ}[1]{\begin{equation*}#1\end{equation*}}

\newcommand{\eqna}[1]{\begin{eqnarray}#1\end{eqnarray}}

\makeatletter

\def\LT@makecaption#1#2#3{%
  \LT@mcol\LT@cols c{\hbox to\z@{\hss\parbox[t]\LTcapwidth{%
    \sbox\@tempboxa{#1{#2. }#3}%
    \ifdim\wd\@tempboxa>\hsize
      #1{#2. }#3%
    \else
      \hbox to\hsize{\hfil\box\@tempboxa\hfil}%
    \fi
    \endgraf\vskip\baselineskip}%
  \hss}}}

\@addtoreset{equation}{section}
\@addtoreset{footnote}{section}

\newenvironment{casks}{%
  \matrix@check\casks\env@casks
}{%
  \endarray\right.%
}
\def\env@casks{%
  \let\@ifnextchar\new@ifnextchar
  \left\lbrack
  \def\arraystretch{1.2}%
  \array{@{}l@{\quad}l@{}}%
}

\newcounter{numt}
\newcounter{col}
\newcounter{coll}

\makeatother

\renewcommand{\ge}{\geqslant}
\renewcommand{\le}{\leqslant}
\newcommand{\fa}{\,\forall\,}

\newcommand{\bes}{\infty}

\newcommand{\eqi}{\equiv}

\newcommand{\subs}{\subset}

\newcommand{\sm}{\setminus}

\newcommand{\cln}{\colon}

\newcommand{\Ra}{\Rightarrow}
\newcommand{\Lra}{\Leftrightarrow}
\newcommand{\xra}{\xrightarrow}

\newcommand{\ol}{\overline}

\newcommand*{\bw}[1]{#1\nobreak\discretionary{}{\hbox{$\mathsurround=0pt #1$}}{}}

\newcommand{\br}[1]{\bigl(#1\bigr)}
\newcommand{\Br}[1]{\Bigl(#1\Bigr)}

\newcommand{\bgm}[1]{\bigl|#1\bigr|}
\newcommand{\Bm}[1]{\Bigl|#1\Bigr|}

\newcommand{\bc}[1]{\bigl\{#1\bigr\}}

\newcommand{\bb}{\bigm/}

\newcommand{\mbb}{\mathbb}
\newcommand{\mbf}{\mathbf}

\newcommand{\R}{\mbb{R}}

\newcommand{\Z}{\mbb{Z}}
\newcommand{\N}{\mbb{N}}
\newcommand{\T}{\mbb{T}}

\newcommand{\Cbb}{\mbb{C}}

\newcommand{\ga}{\gamma}
\newcommand{\Ga}{\Gamma}

\newcommand{\la}{\lambda}

\newcommand{\rh}{\rho}

\newcommand{\ph}{\varphi}

\DeclareMathOperator{\Ker}{Ker}

\newcommand{\GL}{\mbf{GL}}

\newcommand{\bom}{\boldmath}

\newcommand{\thra}{\twoheadrightarrow}

\begin{document}

\author{O.\,G.\?Styrt}
\title{The existence\\
of a~polynomial factorization map\\
for some compact linear groups}
\date{}
\newcommand{\udk}{512.815.1+512.815.6+512.816.1+512.816.2}
\newcommand{\msc}{14L24+22C05}

\maketitle

{\leftskip\parind\rightskip\parind
It is proved that each of compact linear groups of one special type admits a~polynomial factorization map onto a~real vector space. More exactly, the
group is supposed to be non-commutative one-dimensional and to have two connected components, and its representation should be the direct sum of three
irreducible two-dimensional real representations at least two of them being faithful.

\smallskip

\textbf{Key words\:} Lie group, factorization map of an action.\par}

\section{Introduction}\label{introd}

In this paper, we prove that each of the compact linear groups of one certain type admits a~polynomial factorization map onto a~vector space.

This problem arose from the question when the topological quotient of a~compact linear group is homeomorphic to a~vector space that was researched in
\cite{MAMich,My1,My2,My3,My4}. In~\cite{MAMich}, the result for finite groups is obtained. Namely, the necessary and sufficient condition is that the
group is generated by pseudoreflections. Moreover, for the most of such groups, a~polynomial factorization map onto a~vector space is explicitly
constructed. In~\cite{My1}, the case of infinite groups with commutative connected components is considered. For any $n_1,n_2,n_3\in\N$, denote by
$G(n_1,n_2,n_3)$ the subgroup of the group $\GL_{\R}(\Cbb^3)$ generated by the operators $(z_1,z_2,z_3)\bw\to(\la^{n_1}z_1,\la^{n_2}z_2,\la^{n_3}z_3)$
($\la\in\T$) and $(z_1,z_2,z_3)\bw\to(\ol{z}_1,\ol{z}_2,\ol{z}_3)$. In~\cite{My1}, it is proved that $\Cbb^3\bb\br{G(n_1,n_2,n_3)}\cong\R^5$ for all
$n_1,n_2,n_3\in\N$. However, the proof is based on purely topological aspects. In this work, we will prove that each of the linear groups
$G(1,1,n)\bw\subs\GL_{\R}(\Cbb^3)$ ($n\in\N$) admits an $\R$\д polynomial factorization map $\Cbb^3\thra\R^5$.

Fix an arbitrary number $n\in\N$. Set $G:=G(1,1,n)\bw\subs\GL_{\R}(\Cbb^3)$.

The group $G\bw\subs\GL_{\R}(\Cbb^3)$ is generated by the operators $(z_1,z_2,z_3)\bw\to(\la z_1,\la z_2,\la^n z_3)$ ($\la\in\T$) and
$(z_1,z_2,z_3)\bw\to(\ol{z}_1,\ol{z}_2,\ol{z}_3)$. The map $\Cbb^3\bw\to\Cbb^3$, $z\bw\to w$, where $w_1:=z_1\bw+z_2$,
$w_2:=\ol{(z_1-z_2)}\bw=\ol{z}_1\bw+\ol{z}_2$, $w_3:=z_3$, is an $\R$\д linear automorphism. In terms of the so-defined coordinates $w_1,w_2,w_3$, the
linear group $G\bw\subs\GL_{\R}(\Cbb^3)$ is generated by the operators $(w_1,w_2,w_3)\bw\to(\la w_1,\ol{\la}w_2,\la^n w_3)$ ($\la\in\T$) and
$\tau\cln(w_1,w_2,w_3)\bw\to(w_2,w_1,\ol{w}_3)$. The group $G^0\bw\subs\GL_{\R}(\Cbb^3)$ consists of all the operators
$(w_1,w_2,w_3)\bw\to(\la w_1,\ol{\la}w_2,\la^n w_3)$ ($\la\in\T$). Also, $G=G^0\sqcup(G^0\tau)$ and $G/G^0\cong\Z_2$.

In this paper, we will show that the (obviously, $\R$\д polynomial) map
\equ{
\Cbb^3\bw\to\R^5\cong\Cbb^2\oplus\R,\,(w_1,w_2,w_3)\to\Br{w_2^nw_3+w_1^n\ol{w}_3,w_1w_2,|w_3|^2-\br{|w_1|^{2n}+|w_2|^{2n}}}}
is a~factorization map of the linear group $G\bw\subs\GL_{\R}(\Cbb^3)$.

\section{Proof of the main result}

Consider a~group $\Ga=\{e,\ga_0,\ga',\ga\"\}\cong\Z_2\times\Z_2$, the set $C:=\R_{\ge0}^3\bw\times\Cbb^3$, and the action $\Ga\cln C$ uniquely defined by
the relations $\ga'p:=(r_2,r_1,r_3,v_1,v_2,v_3)$ and $\ga\"p:=(r_1,r_2,r_3,v_2,v_1,v_3)$ ($p\bw=(r_1,r_2,r_3,v_1,v_2,v_3)\bw\in C$). Denote by~$\Ga_0$
the subgroup $\{e,\ga_0\}\subs\Ga$. Since $G/G^0\bw\cong\Z_2\bw\cong\Ga_0$, there exists a~homomorphism $\rh\cln G\thra\Ga_0$ such that $\Ker\rh=G^0$ and
$\rh(\tau)=\ga_0$. The subset $M\subs C$ of all elements $(r_1,r_2,r_3,v_1,v_2,v_3)\bw\in C$ such that $|v_1|=r_2^nr_3$, $|v_2|=r_1^nr_3$,
$|v_3|=r_1r_2$, $v_1v_2=r_3^2v_3^n$ is $\Ga_0$\д invariant.

The map $\pi_0\cln\Cbb^3\bw\to C,\,(w_1,w_2,w_3)\bw\to\br{|w_1|,|w_2|,|w_3|,w_2^nw_3,w_1^n\ol{w}_3,w_1w_2}$ satisfies $\pi_0(\Cbb^3)\bw=M\bw\subs C$, and
its fibres coincide with the orbits of the action $G^0\cln\Cbb^3$. It is easy to see that $\pi_0\circ\tau\eqi\ga_0\circ\pi_0$. Thus,
\eqn{\label{comm}
\begin{array}{ll}
\fa g\in G\quad\quad\quad&\pi_0\circ g\eqi\rh(g)\circ\pi_0;\\
\fa x,y\in\Cbb^3\quad\quad\quad&(y\in Gx)\quad\Lra\quad\Br{\pi_0(y)\in\Ga_0\br{\pi_0(x)}}.
\end{array}}

\begin{stm} If $p,q\in M$ and $q\in\Ga p$, then $q\in\Ga_0p$.
\end{stm}

\begin{proof} We have $q=\ga p$ for some $\ga\in\Ga$. If $\ga\in\Ga_0$, then the claim is evident. Suppose that $\ga\notin\Ga_0$. From the relations
$|\Ga/\Ga_0|=2$ and $\ga\notin\Ga_0$, we obtain that $\Ga_0\ga=\Ga\sm\Ga_0=\{\ga',\ga\"\}$, $\Ga_0q\bw=\Ga_0\ga p\bw=\{\ga'p,\ga\"p\}$. Further,
$\Ga_0M=M$ and $q\in M$, implying $\ga'p\bw\in\Ga_0q\bw\subs\Ga_0M\bw=M$. The element $p\in M$ has the form $p\bw=(r_1,r_2,r_3,v_1,v_2,v_3)$, where
$r_1,r_2,r_3\bw\in\R_{\ge0}$, $v_1,v_2,v_3\bw\in\Cbb$, $|v_1|=r_2^nr_3$, $|v_2|=r_1^nr_3$, $|v_3|=r_1r_2$, $v_1v_2=r_3^2v_3^n$. Also,
$(r_2,r_1,r_3,v_1,v_2,v_3)=\ga'p\bw\in M$ and, hence, $r_1^nr_3=|v_1|=r_2^nr_3$, $(r_1^n-r_2^n)r_3=0$.

\begin{cas} $r_3=0$.
\end{cas}

We have $|v_1|=r_2^nr_3=0$ and $|v_2|=r_1^nr_3=0$. Thus, $v_1=v_2=0$, $p=\ga\"p\in\Ga_0q$, $q\in\Ga_0p$.

\begin{cas} $r_3\ne0$.
\end{cas}

Since $(r_1^n-r_2^n)r_3=0$, we have $r_1^n=r_2^n$, $r_1=r_2$, $p=\ga'p\in\Ga_0q$, $q\in\Ga_0p$.
\end{proof}

\begin{imp}\label{gaga} For any $p,q\in M$, we have $(q\in\Ga p)\Lra(q\in\Ga_0p)$.
\end{imp}

Denote by~$\pi$ the map
\equ{
C\to C,\quad\quad\quad(r_1,r_2,r_3,v_1,v_2,v_3)\to\br{(r_1^n-r_2^n)^2,r_1r_2,r_3^2,v_1+v_2,v_1v_2,v_3}.}

\begin{stm}\label{sp} The map $\Cbb^2\to\Cbb^2,\,(v_1,v_2)\to(v_1+v_2,v_1v_2)$ is surjective, and its fibres coincide with the subsets
$\bc{(v_1,v_2),(v_2,v_1)}\subs\Cbb^2$.
\end{stm}

We omit the proof since it is clear.

It is easy to see that, for any $m\in\N$, the map $\R_{\ge0}^2\to\R\times\R_{\ge0},\,(r_1,r_2)\to(r_1^m-r_2^m,r_1r_2)$ is a~bijection. This implies the
following statement.

\begin{stm}\label{dp} For any $m\in\N$, the map $\R_{\ge0}^2\to\R_{\ge0}^2,\,(r_1,r_2)\to\br{(r_1^m-r_2^m)^2,r_1r_2}$ is surjective, and its fibres
coincide with the subsets $\bc{(r_1,r_2),(r_2,r_1)}\subs\R_{\ge0}^2$.
\end{stm}

\begin{imp}\label{piga} The map $\pi\cln C\bw\to C$ is surjective, and its fibres coincide with the orbits of the action $\Ga\cln C$.
\end{imp}

\begin{proof} Follows from Statements \ref{sp} and~\ref{dp}.
\end{proof}

\begin{prop}\label{pip} The map $(\pi\circ\pi_0)\cln\Cbb^3\bw\to C$ satisfies $(\pi\circ\pi_0)(\Cbb^3)\bw=\pi(\Ga M)\bw\subs C$, and its fibres coincide
with the orbits of the action $G\cln\Cbb^3$.
\end{prop}

\begin{proof} Recall that $\pi_0(\Cbb^3)\bw=M\bw\subs C$.

By Corollary~\ref{piga}, $\pi(M)=\pi(\Ga M)$. Thus, $(\pi\circ\pi_0)(\Cbb^3)\bw=\pi\br{\pi_0(\Cbb^3)}\bw=\pi(M)\bw=\pi(\Ga M)$.

It remains to prove that $(y\in Gx)\Lra\br{(\pi\circ\pi_0)(x)=(\pi\circ\pi_0)(y)}$ for all $x,y\in\Cbb^3$.

Take arbitrary elements $x,y\in\Cbb^3$. Since $\pi_0(\Cbb^3)\bw=M\bw\subs C$, we have $p:=\pi_0(x)\in M\bw\subs C$ and $q:=\pi_0(y)\in M\bw\subs C$. It
follows from Corollary~\ref{gaga} that $(q\in\Ga p)\Lra(q\in\Ga_0p)$. Further, by Corollary~\ref{piga}, $(q\in\Ga p)\Lra\br{\pi(p)=\pi(q)}$. According
to~\eqref{comm}, $(y\in Gx)\Lra(q\in\Ga_0p)$. So, $(y\in Gx)\Lra(q\in\Ga_0p)\Lra(q\in\Ga p)\Lra\br{\pi(p)=\pi(q)}$, $(y\in Gx)\Lra\br{\pi(p)=\pi(q)}$.
\end{proof}

Consider the subset $L\subs C$ of all elements $(s_1,s_2,s_3,u_1,u_2,u_3)\bw\in C$ such that $|u_3|=s_2$, $u_2=s_3u_3^n$,
$|u_1^2|+|u_1^2-4u_2|-|4u_2|=2s_1s_3$.

\begin{lemma}\label{pil} One has $\pi^{-1}(L)=\Ga M$.
\end{lemma}

\begin{proof} We should show that $(p\in\Ga M)\Lra\br{\pi(p)\in L}$ for any $p\bw\in C$.

Take an arbitrary element $p\bw=(r_1,r_2,r_3,v_1,v_2,v_3)\bw\in C$.

We have $\ga'p=(r_2,r_1,r_3,v_1,v_2,v_3)\bw\in C$ and $\pi(p)=(s_1,s_2,s_3,u_1,u_2,u_3)\bw\in C$, where $s_1:=(r_1^n-r_2^n)^2$, $s_2:=r_1r_2$,
$s_3:=r_3^2$, $u_1:=v_1+v_2$, $u_2:=v_1v_2$, $u_3:=v_3$.

Note that $\Ga=\Ga_0\cup\ga'\Ga_0$. Also, $\Ga_0M=M$. Hence, $\Ga M=(\Ga_0M)\cup(\ga'\Ga_0M)=M\cup(\ga'M)$,
\eqn{\label{gam}
(p\in\Ga M)\quad\Lra\quad\br{(p\in M)\lor(p\in\ga'M)}\quad\Lra\quad\br{(p\in M)\lor(\ga'p\in M)}.}

Clearly, $p\in M$ if and only if $|v_1|=r_2^nr_3$, $|v_2|=r_1^nr_3$, $|v_3|=r_1r_2$, $v_1v_2=r_3^2v_3^n$. Further, $\ga'p\in M$ if and only if
$|v_1|=r_1^nr_3$, $|v_2|=r_2^nr_3$, $|v_3|=r_1r_2$, $v_1v_2=r_3^2v_3^n$. By~\eqref{gam}, $p\in\Ga M$ if and only if
\begin{nums}{-1}
\item\label{com} $|u_3|=s_2$, $u_2=s_3u_3^n$\~
\item\label{dif} $\br{|v_1|,|v_2|}=(r_2^nr_3,r_1^nr_3)\in\R_{\ge0}^2$ or $\br{|v_1|,|v_2|}=(r_1^nr_3,r_2^nr_3)\in\R_{\ge0}^2$.
\end{nums}
According to Statement~\ref{dp}, the condition~\ref{dif} holds if and only if $\br{|v_1|-|v_2|}^2=(r_1^nr_3-r_2^nr_3)^2$ and
$|v_1|\cdot|v_2|=(r_1^nr_3)(r_2^nr_3)$. We have $|v_1|\cdot|v_2|=|u_2|$, $(r_1^nr_3)(r_2^nr_3)=s_2^ns_3$, $(r_1^nr_3-r_2^nr_3)^2\bw=s_1s_3$. Also,
$(v_1-v_2)^2=(v_1+v_2)^2-4v_1v_2=u_1^2-4u_2$,
\begin{gather*}
2\br{|v_1|-|v_2|}^2=2\br{|v_1|^2+|v_2|^2}-4\br{|v_1|\cdot|v_2|}=|v_1+v_2|^2+|v_1-v_2|^2-4|u_2|=\\
=\bgm{(v_1+v_2)^2}+\bgm{(v_1-v_2)^2}-4|u_2|=|u_1^2|+|u_1^2-4u_2|-|4u_2|.
\end{gather*}
Therefore, the condition~\ref{dif} holds if and only if $|u_1^2|+|u_1^2-4u_2|-|4u_2|=2s_1s_3$ and $|u_2|=s_2^ns_3$. Thus, $p\in\Ga M$ if and only if
$|u_3|=s_2$, $u_2=s_3u_3^n$, $|u_1^2|+|u_1^2-4u_2|-|4u_2|=2s_1s_3$, $|u_2|=s_2^ns_3$. Note that $\br{(|u_3|=s_2)\land(u_2=s_3u_3^n)}\Ra(|u_2|=s_2^ns_3)$.
Hence, $(p\in\Ga M)\Lra\br{\pi(p)\in L}$.
\end{proof}

\begin{imp}\label{pigm} One has $\pi(\Ga M)=L$.
\end{imp}

\begin{proof} According to Corollary~\ref{piga}, the map $\pi\cln C\bw\to C$ is surjective. Hence, $\pi\br{\pi^{-1}(L)}=L$. By Lemma~\ref{pil},
$\pi^{-1}(L)=\Ga M$, $\pi(\Ga M)=\pi\br{\pi^{-1}(L)}=L$.
\end{proof}

\begin{imp}\label{pp} The map $(\pi\circ\pi_0)\cln\Cbb^3\bw\to C$ satisfies $(\pi\circ\pi_0)(\Cbb^3)\bw=L\bw\subs C$, and its fibres coincide with the
orbits of the action $G\cln\Cbb^3$.
\end{imp}

\begin{proof} Follows from Proposition~\ref{pip} and Corollary~\ref{pigm}.
\end{proof}

Consider the maps
\equ{
\begin{aligned}
\ph_0&\cln&L&\to\Cbb^2\oplus\R,&(s_1,s_2,s_3,u_1,u_2,u_3)&\to(u_1,u_3,s_3-s_1);\\
\ph&\cln&\Cbb^2\oplus\R&\to\Cbb^2\oplus\R,&(a,b,c)&\to\br{a,b,c-2|b|^n};\\
(\ph\circ\ph_0)&\cln&L&\to\Cbb^2\oplus\R,&(s_1,s_2,s_3,u_1,u_2,u_3)&\to\br{u_1,u_3,s_3-s_1-2|u_3|^n}.
\end{aligned}}

\begin{lemma}\label{fib} The map $\ph_0\cln L\to\Cbb^2\oplus\R$ is a~bijection.
\end{lemma}

\begin{proof} We need to prove that $\bgm{\ph_0^{-1}(a,b,c)}=1$ for any $(a,b,c)\in\Cbb^2\oplus\R$.

Take an arbitrary element $(a,b,c)\in\Cbb^2\oplus\R$.

The subset $\ph_0^{-1}(a,b,c)\subs L$ consists exactly of all elements $(t-c,|b|,t,a,tb^n,b)\bw\in C$ ($t\in\R_{\ge0}$, $t\ge c$,
$|a^2|+|a^2-4tb^n|-|4tb^n|=2(t-c)t$).

Set $f_+(t):=2(t-c)t$, $f_-(t):=|a^2|+|a^2-4tb^n|-|4tb^n|$, and $f(t):=f_+(t)-f_-(t)$ ($t\in\R$). We should show that there exists a~unique number
$t\in\R_{\ge0}$ such that $t\ge c$ and $f(t)=0$. The subset $I:=\{t\in\R_{\ge0}\cln t\ge c\}\subs\R_{\ge0}$ has the form $[t_0;+\bes)$,
$t_0:=\max\{0,c\}\in\R_{\ge0}$.

The function~$f_+$ is strictly increasing on the subset $I\subs\R_{\ge0}$ and satisfies
\eqna{
&&f_+(t_0)=0;\label{zer}\\
&&f_+(t)\xra[t\to\bes]{}+\bes.\label{inf}}

Obviously, $f_+,f_-\in C(\R)$, implying $f\in C(\R)$.

Due to the triangle inequality, $f_-(t)=|a^2|+|a^2-4tb^n|-|4tb^n|\ge0$ ($t\in\R$). In particular, $f_-(t_0)\ge0$. By~\eqref{zer}, $f(t_0)\le0$.

If $t',t\"\in I$ and $t'\le t\"$, then $t'\in I\subs\R_{\ge0}$, $0\le t'\le t\"$, $|4t\"b^n|-|4t'b^n|=|4t\"b^n-4t'b^n|$ and, due to the triangle
inequality,
\begin{align*}
f_-(t')-f_-(t\")&=\br{|a^2|+|a^2-4t'b^n|-|4t'b^n|}-\br{|a^2|+|a^2-4t\"b^n|-|4t\"b^n|}=\\
&=|a^2-4t'b^n|-|a^2-4t\"b^n|+\br{|4t\"b^n|-|4t'b^n|}=\\
&=|a^2-4t'b^n|+|4t\"b^n-4t'b^n|-|a^2-4t\"b^n|\ge0,
\end{align*}
$f_-(t')\ge f_-(t\")$. We see that the function~$f_-$ is decreasing on the subset $I\subs\R_{\ge0}$. By~\eqref{inf}, $f(t)\xra[t\to+\bes]{}+\bes$.

On the subset $I\subs\R_{\ge0}$, the function~$f_+$ is strictly increasing and the function~$f_-$ is decreasing. Hence, the function~$f$ is strictly
increasing on the subset $I=[t_0;+\bes)\subs\R_{\ge0}$. Also, $f\in C(\R)$, $f(t_0)\le0$, and $f(t)\xra[t\to+\bes]{}+\bes$. Thus,
$\Bm{I\cap\br{f^{-1}(0)}}=1$ as required.
\end{proof}

\begin{imp}\label{ff} The map $(\ph\circ\ph_0)\cln L\to\Cbb^2\oplus\R$ is a~bijection.
\end{imp}

\begin{proof} One can easily see that the map $\ph\cln\Cbb^2\oplus\R\to\Cbb^2\oplus\R$ is a~bijection. It remains to apply Lemma~\ref{fib}.
\end{proof}

By Corollary~\ref{pp}, the map $(\pi\circ\pi_0)\cln\Cbb^3\bw\to C$ can be considered as the surjective map $(\pi\circ\pi_0)\cln\Cbb^3\bw\thra L$ whose
fibres coincide with the orbits of the action $G\cln\Cbb^3$. According to Corollary~\ref{ff}, the map
$(\ph\circ\ph_0\circ\pi\circ\pi_0)\cln\Cbb^3\bw\to\Cbb^2\oplus\R$ is surjective, and its fibres coincide with the orbits of the action $G\cln\Cbb^3$.
This means that the map $(\ph\circ\ph_0\circ\pi\circ\pi_0)\cln\Cbb^3\bw\to\Cbb^2\oplus\R$ is a~factorization map of the linear group
$G\bw\subs\GL_{\R}(\Cbb^3)$. If $w=(w_1,w_2,w_3)\in\Cbb^3$, then
\equ{\begin{array}{c}
\begin{aligned}
\pi_0(w)&=\br{|w_1|,|w_2|,|w_3|,w_2^nw_3,w_1^n\ol{w}_3,w_1w_2};\\
(\pi\circ\pi_0)(w)&=\Br{\br{|w_1|^n-|w_2|^n}^2,|w_1|\cdot|w_2|,|w_3|^2,w_2^nw_3+w_1^n\ol{w}_3,(w_1w_2)^n\cdot|w_3|^2,w_1w_2};
\end{aligned}\\
\begin{aligned}
(\ph\circ\ph_0\circ\pi\circ\pi_0)(w)&=\Br{w_2^nw_3+w_1^n\ol{w}_3,w_1w_2,|w_3|^2-\br{|w_1|^n-|w_2|^n}^2-2|w_1w_2|^n}=\\
&=\Br{w_2^nw_3+w_1^n\ol{w}_3,w_1w_2,|w_3|^2-\br{|w_1|^{2n}+|w_2|^{2n}}}.
\end{aligned}
\end{array}}
Thus, the map $(\ph\circ\ph_0\circ\pi\circ\pi_0)\cln\Cbb^3\bw\to\Cbb^2\oplus\R$ is $\R$\д polynomial.

We see that the linear group $G\bw\subs\GL_{\R}(\Cbb^3)$ admits an $\R$\д polynomial factorization map
$(\ph\circ\ph_0\circ\pi\circ\pi_0)\cln\Cbb^3\bw\to\Cbb^2\oplus\R$. This completely proves the main result of the paper.


\newpage

\begin{thebibliography}{9}
\bibitem{MAMich}
M.\,A.\?Mikhailova, \textit{On the quotient space modulo the action of a finite group generated by pseudoreflections}, Mathematics of the USSR-Izvestiya,
1985, vol.\,24, \No1, 99---119.
\bibitem{My1}
O.\,G.\?Styrt, \textit{On the orbit space of a compact linear Lie group with commutative connected component}, Tr. Mosk. Mat. O-va, 2009, vol.\,70,
235---287 (Russian).
\bibitem{My2}
O.\,G.\?Styrt, \textit{On the orbit space of a three-dimensional compact linear Lie group}, Izv. RAN, Ser. math., 2011, vol.\,75, \No4, 165---188
(Russian).
\bibitem{My3}
O.\,G.\?Styrt, \textit{On the orbit space of an irreducible representation of a special unitary group}, Tr. Mosk. Mat. O-va, 2013, vol.\,74, \No1,
175---199 (Russian).
\bibitem{My4}
O.\,G.\?Styrt, \textit{On the orbit spaces of irreducible representations of simple compact Lie groups of types $B$, $C$, and~$D$}, J.~Algebra, 2014,
vol.\,415, 137---161.
\end{thebibliography}
\end{document}